\documentclass[10pt,a4paper]{article}
\usepackage{amssymb}
\usepackage{amsmath}
\usepackage{amsthm}
\usepackage{float}
\usepackage[T1]{fontenc}
\usepackage[utf8]{inputenc}
\usepackage[british]{babel}
\usepackage{geometry}
\usepackage{caption}
\usepackage{graphicx}
\usepackage{color}
\usepackage{nomencl}
\makeglossary
\makenomenclature

\newcommand{\Hyp}{\mathbb{H}}

\newcommand{\N}{\mathbb{N}}

\newcommand{\G}{\Gamma}
\newcommand{\g}{\gamma}

\newcommand{\tv}{\rightarrow}

\newcommand{\Sym}{\mathrm{Sym}}
\newcommand{\Jac}{\mathrm{Jac}}
\newcommand{\bary}{\mathrm{bar}}

\newtheorem{theorem}{Theorem}
 
\newtheorem{lemme}[theorem]{Lemma}

\newtheorem{prop}[theorem]{Proposition}

\newtheorem{defi}[theorem]{Definition}

\DeclareMathOperator{\Isom}{Isom}
\DeclareMathOperator{\Card}{Card}

\DeclareMathOperator{\Id}{Id}

\DeclareMathOperator{\Tr}{Tr}

\DeclareMathOperator{\vol}{Vol}

\title{Convergence of quasi-Fuchsian groups using critical exponent}
\author{Olivier Glorieux}
\begin{document}
\maketitle
\begin{abstract}
We prove that a sequence of quasi-Fuchsian representations  for which the critical exponent converges to the topological dimension of the boundary of the group (larger than 2),  converges up to subsequence and conjugacy to a totally geodesic representation. 
\end{abstract}

\section{Introduction}
Given $\G$ a cocompact lattice of $\Isom(\Hyp^m)$ and a totally geodesic copy of $\Hyp^m$ into $\Hyp^n$, $n>m$, we can see $\G$ as a discrete group of $\Isom(\Hyp^n)$. Indeed   the isometry group of $\Hyp^m$  can naturally be seen as a subgroup of $\Isom(\Hyp^n)$ preserving the totally geodesic copy of $\Hyp^m \subset \Hyp^n$. We call this representation $\rho_0\, :\,  \G\tv \Isom(\Hyp^n)$ a \emph{Fuchsian} representation. If one choose another copy of $\Hyp^m$ inside $\Hyp^n$, the new Fuchsian representation is conjugated by an element of $\Isom(\Hyp^n)$ to $\rho_0$. By Mostow rigidity, if $m\geq 3$, every representations of $\G$ inside $\Isom(\Hyp^m)$ are conjugated by an element of $\Isom(\Hyp^m)$. So all discrete, faithful and totally geodesic representations of $\G$  into $\Isom(\Hyp^n)$ are conjugated to $\rho_0$, if $m\geq 3$. If $m=2$, there exists non conjugate representations of $\G$ inside $\Isom(\Hyp^2)$, this is the Teichmüller space of $\G$. We will suppose for the rest of the paper that $m\geq 3$.   \\
%The limit set of this Fuchsian  representation is a round sphere : the geometric boundary of $\Hyp^m$. For $m\geq 3$, thanks to  Mostow rigidity, the only deformation of $\G$ inside $\Isom(\Hyp^m)$ are conjugations. However, it can be non trivially deform inside $\Isom(\Hyp^n)$. A representation $\rho \,  :\, \G \tv \Isom(\Hyp^n)$ whose limit set is a topological sphere, is called \emph{quasi-Fuchsian}. \\ 
%A quasi-Fuchsian representation $\rho$ is conjugated to $\rho_0$ if and only if the limit sphere is a round sphere, as soon as $m\geq 3$. The hypothesis, $m\geq 3$ cannot be avoided because of the Teichmüller space: we can have a sequence of representation which fixes a totally geodesic copy of $\Hyp^2\subset \Hyp^n$ but which is not conjugated to $\rho_0$. \\

Even if the representation $\rho_0$ cannot be deform inside $\Isom(\Hyp^m)$, there exists discrete  and faithful deformations of $\rho_0$ in $\Isom(\Hyp^n)$ which are not anymore Fuchsian, ie. which does not preserve any totally geodesic copy of $\Hyp^m$. We call them \emph{quasi-Fuchsian} representations.\\
There is a numerical invariant which measures how far from being Fuchsian a representation is; it is called the \emph{critical exponent } and defined in the following way: 
$$\delta(\rho(\G)) :=\limsup_{R\tv +\infty} \frac{1}{R}\log\Card \{\g\in \G \, |\, d(\rho(\g) x,x)\leq R\},$$
it is independent of the base point $x\in\Hyp^n$ thanks to the triangle inequality. It measure the exponential growth rate of an orbit inside $\Hyp^n$.

By a simple computation using volume of balls, we can see that $\delta(\rho_0(\G))=m-1$. In fact, a theorem of C. Yue, \cite{yue1996dimension} shows that critical exponent distinguishes Fuchsian representation:  a quasi-Fuchsian representation $\rho$ is conjugated to $\rho_0$ if and only if $\delta(\rho(\G))=m-1$. Later, Besson-Courtois-Gallot, \cite[Theorem 1.14]{besson1999lemme} showed that the convex-cocompact hypothesis is not needed, and proved that if $\rho,$ a discrete and faithful representation of $\G$, satisfies $\delta(\rho(\G)) =m-1$ then $\rho $ si conjugated to $\rho_0$. (In \cite{besson1999lemme} the theorem is cited with the convex-cocompact hypothesis, however they explained just after that the hypothesis is not needed).

\paragraph{Remark} In dimension $2$ the corresponding statement is $\delta(\rho(\G))$ is equal to $1$ if and only if it preserves a totally geodesic copy, but it is not necessarily conjugated to $\rho_0$. However in this dimension, a lot of work has been done, and we know some examples where we can compute the limit of the critical exponent for a sequence of quasi-Fuchsian representations, see \cite{mcmullen1999hausdorff}. Moreover the work of A. Sanders, \cite{sanders2014entropy} shows that for a sequence of quasi-Fuchsian representations (if we suppose that the injectivity radius is bounded below) if the critical exponent goes to $1$ then the sequence is close to a totally geodesic one. The aim of this article is to show a corresponding result in higher dimension, and we can even obtain convergence due to the absence of non trivial deformations inside $\Isom(\Hyp^m)$. 

\begin{theorem}\label{th principal}
Let $m\geq 3$ and $\rho_j$ be a sequence of quasi-Fuchsian representations. If $\delta(\rho_j(\G))\tv m-1$  then up to subsequence and conjugacy $\rho_j$ converges to $\rho_0$.
\end{theorem}

Let us make some comments. Usually theorems often go in the opposite direction: we suppose that the sequence of groups converges (algerically, geometrically or strongly)  and give a result on the continuity of critical exponent. (Even the result of A. Sanders, does not show convergence.) For example, if one knows that $\rho_j$ converges algebraically to a \emph{convex cocompact } representation $\rho_\infty$, then a theorem of McMullen \cite[Theorem 7.1]{mcmullen1999hausdorff}, implies that $\delta(\rho_\infty(\G)) =m-1$ and hence by  Yue's Theorem  \cite{yue1996dimension}, we know that $\rho_\infty$ is conjugated to $\rho_0$. The fact that we know the geometric structure of the limit representation is very important in the Theorem of McMullen. He explained in his paper how we can obtain sequence of representations for which the critical exponent is not continuous. \\
However in our case, putting together some deep theorems, we can show it is sufficient to prove that $\rho_k$ converges algebraically to some representation $\rho_\infty$, for Theorem \ref{th principal} to be true. Comparing to McMullen's work, here we do not need to know that the limit is convex cocompact, or that there is strong convergence. 
\begin{prop}
Let $\rho_k$ be a sequence of discrete and faithful representations in $\Isom(\Hyp^n)$, converging algebraically to $\rho_\infty$. Suppose that $\delta(\rho_k(\G))\tv m-1$ then $\delta(\rho_\infty(\G))=m-1$ and therefore, $\rho_\infty$  is conjugated to $\rho_0$. 
\end{prop}

\begin{proof}

First we use a theorem of Kapovich, \cite[Theorem 1.1]{kapovich2008sequences}, saying that  if a sequence of discrete and faithful representations in  $\Isom(\Hyp^n)$ converges, then the limit is discrete and faithful. 
Moreover a result of Bishop-Jones  \cite{BishopJones} says that the critical exponent is lower semi-continuous, therefore  $$\delta(\rho_\infty) \leq \liminf \delta(\rho_k(\G)) =m-1.$$
We conclude by the Theorem of Besson-Courtois-Gallot previsously cited, to conclude that $\rho_\infty$ is conjugated to $\rho_0$. 
\end{proof}
 
In  their article Bishop-Jones give the result for subgroups of $\Isom(\Hyp^3)$, however their proof work in any dimension. \\

Therefore, the task is to show that under the critical exponent hypothesis the sequence of representations $\rho_k$ converges algebraically to some representation. For this we will adapt the construction of Besson, Courtois, Gallot  \cite{besson2007inegalites}.

 \section{The Besson, Courtois, Gallot construction}
%We will adapt the proof of Besson, Courtois and Gallot, \cite{besson2007inegalites} to show the convergence.
 Let $Y= \Hyp^m$ and $X=\Hyp^n$. We are going to recall their construction of a sequence of maps  $F_j \, : \, Y\tv X $, $(\G,\rho_j(\G))$-equivariant  for which we can control the Jacobian and shows that it converges (up to subsequence and conjugation).  
 
The maps $F_j$ are the compositions of the following two:
\begin{itemize}
\item The first is  the map $y\tv \mu_y$ which goes from $Y$ to  the set of finite measures on $\partial X$. It associates to  a point $y$ the push forward of the Patterson-Sullivan measure $(\nu_y)$ on $\partial Y$ by an equivariant homomophism $f_j$  from $\Lambda(\G)=\partial Y$ to $\Lambda(\rho_j(\G))\subset \partial X $. We normalize $\mu_y$ into a probability measure. 
\item The second is the barycenter map going from the set of finite measures on $\partial X$ to the space $X$. It associates to a measure $\mu$, the unique point $\bary (\mu) $, where the function: 
$$\mathcal{B} \, : \, x\tv \int_{\partial X} \beta_X(\xi, x) d\mu(\xi),$$
reaches its minimum. Here $\beta_X(\xi,x)$ is the Busemann function on $X$, normalized by taking an origin $o\in X$. It is shown in \cite{besson1995entropies} that $\bary(\mu)$ is well defined as soon as $\mu$ has no atoms whose measure is greater than $\frac{1}{2}\mu(\partial X)$. 
\end{itemize}
We define the map $F_j$ by $$F_j(y) := \bary(\mu_y).$$
 The  Patterson-Sullivan density  on $Y$, $\nu_y$,  satisfies, $\nu_{\g y} = \g_*(\nu_y)$ for all $\g\in \G$. The barycenter map satisfies $\bary(\g_* \mu) = \g(\bary(\mu))$, for all $\g \in \Isom(X)$. Therefore, the maps $F_j$ are $(\G,\rho_j(\G))$-equivariant.

Following  \cite{besson1999lemme,besson2007inegalites}, we introduce the quadratic forms, $k_y, h_y$ and $h'_y$ defined on $T_{F_j(y)} X$, $T_{F_j(y)} X$ and $T_y Y$, by 
\begin{eqnarray*}
k_{y,j}(v,w)& =&\int_{\partial Y} Dd\beta_X\left|_{(F_j(y),f_j(\xi))} (v,w) d\nu_y(\xi)\right.\\
h_{y,j}(v,w) &=&\int_{\partial Y} d\beta_X\left|_{(F_j(y),f_j(\xi))} (v) d\beta_X\left|_{(F_j(y),f(\xi))} (w) d\nu_y(\xi)\right.\right.\\
h'_y(u,t)& =&\int_{\partial Y} d\beta_Y\left|_{(y,\xi)}(u) d\beta_Y\left|_{(y,\xi)}(t)  d\nu_y(\xi).\right.\right. 
\end{eqnarray*}
for all $v,w\in T_{F_j(y)} X$ and $u,t\in T_y Y$. We denote by $K_{y,j}, H_{y,j} $ and $H'_{y}$ the corresponding symmetric endomorphisms. Note in particular that $h'$ is independent of $j$ and invariant by $\G$, therefore, there exists $C>0$ independent of $y$ and $j$ such that $\| H'\|\leq C$. \\

In order to prove that $F_j$ converges, we will study the behavior of these quadratic forms. We list the principal properties that they satisfy:\\
Since $\nu_y$ is normalized into a probability and $\| d\beta_X\|_{X}=\| d\beta_Y\|_{Y}=1$ we have:
$$\Tr(H_{y,j}) \leq 1$$
$$\Tr(H'_{y}) \leq 1.$$
The implicit functions theorem gives that $F_j$ satisfies: 
$$\int_{\partial Y} Dd{\beta_X}|_{(F_j(y),f(\xi))} (\cdot, d_y F_j(u))d\nu_y(\xi) = (m-1)\int_{\partial Y} d{\beta_X}|_{(F_j(y),f(\xi))} (\cdot) d{\beta_Y}|_{(y,\xi)} (u) d\nu_y(\xi).$$
Now the Cauchy-Schwarz inequality applied on the second member of this equation gives:
\begin{eqnarray}\label{eq - Cauchy Schwarz}
k_{y,j}(v,dF_j(u)) \leq (m-1) h_{y,j}(v,v)^{1/2} h'_{y} (u,u)^{1/2}.
\end{eqnarray}
%Since $\G$ is a cocompact lattice of $Y=\Hyp^m$, we have $\delta(\G)=m-1$: 
%$$k_{y,j}(v,dF_j(u)) \leq (m-1)h_{y,j}(v,v)^{1/2} h'_{y,j} (u,u)^{1/2}.$$
\begin{defi}
The $p$-Jacobian of a function $F:Y\tv X$ is defined by 
$$\Jac_p F(y) = \sup \| dF_y(u_1)\wedge \dots\wedge dF_y(u_p)\|,$$
where the supremum is taken over all $p$-orthonormal frames of $T^1_y Y$. 

When $p=m=dim(Y)$ we will write $\Jac F$.
\end{defi}

By considering an orthonormal basis on $T_{F_{j}(y)}X$, it gives the following inequality on the determinants: 
\begin{eqnarray}
\det(\tilde{K}_{y,j}) \Jac F_j(y) \leq {(m-1)}^{m} \det(\tilde{H}_{j,y})^{1/2} \det (H'_{y} )^{1/2},
\end{eqnarray}
where $\tilde{K}_{y,j}$ and $\tilde{H}_{y,j}$ designed the restriction to $dF_j(T^1_y(Y))\subset T^1_{F(y)} X$  of $K_{y,j}$ and $H_{y,j}$. 
Since $\det (H'_{y} )\leq \left(\frac{1}{m}\Tr(H'_{y}) \right)^m = \frac{1}{m^m},$ we have 
\begin{eqnarray}
\Jac F_j(y) \leq  \frac{(m-1)^m}{m^{m/2}} \frac{\det(\tilde{H}_{j,y})^{1/2}}{\det(\tilde{K}_{y,j})}
\end{eqnarray}
Since $X$ is the hyperbolic space of constant curvature, by direct computations we obtain that: $Dd\beta_X=g_X - d\beta_X \otimes d\beta_X$ therefore, $k_{y,j}= g_X -h_{y,j},$ and then 
\begin{eqnarray}
\det \tilde{K}_{y,j}=\det(\Id - \tilde{H}_{y,j})
\end{eqnarray}

We conclude as in \cite{besson1999lemme} that:
\begin{eqnarray}
\Jac F_j(y) \leq \frac{(m-1)^m}{m^{m/2}}\frac{\det(\tilde{H}_{j,y})^{1/2}}{\det(\Id -\tilde{H}_{j,y})}
\end{eqnarray}

\paragraph{Fact} The map $H\tv \frac{\det (H)^{1/2}}{\det(\Id -H)}$ defined on positive definite symmetric matrices of dimension $m\geq 3$ and trace less than $1$ achieves its unique maximum on $H=\frac{1}{m}\Id$. The value of this maximum is $\frac{m^{m/2}}{(m-1)^m}$.

\noindent Therefore we have:
\begin{eqnarray}
\Jac F_j(y) \leq \left(\frac{(m-1)}{m-1}\right)^m=1
\end{eqnarray}
%The last inequality comes from the fact that $\delta(\G)=m-1$ since $\G$ is a cocompact lattice of $\Hyp^m$.

\noindent Thanks to the previous fact,  Besson-Courtois-Gallot, proved in  \cite{besson1995entropies}:
\begin{lemme}\cite{besson1995entropies}\label{lem - rigid BCG 1}
$ \text{If } \, \Jac F_j(y) =1 \text{ then }\,$ 
$$\frac{\det \tilde{H}_{j,y}^{1/2}}{\det(\Id-\tilde{H}_{j,y})}=\frac{m^{m/2}}{(m-1)^m} \quad \text{and} \quad  \tilde{H}_{j,y}=\frac{1}{m}\Id.$$
\end{lemme}

The following approximation is clear: 
\begin{lemme}\label{lem - approx BCG 1}
$\forall \eta>0, \, \exists \epsilon_0>0, \, \forall 0<\epsilon\leq \epsilon_0,\, $
$ \text{if } \, \left|\Jac F_j -1 \right|\leq \epsilon \text{ then } $
$$\left|\frac{(\det \tilde{H}_{y,j})^{1/2}}{\det(\Id-\tilde{H}_{y,j})}-\frac{m^{m/2}}{(m-1)^m}\right|\leq \eta $$
\end{lemme}
We can also obtain an approximation of the second part of Lemma \ref{lem - rigid BCG 1}. Indeed their proof shows that there is no maximum  of $\frac{(\det H)^{1/2}}{\det(\Id-H)}$ on the boundary $\{\Sym_p^{++}, \cap  \Tr(H)=1 \}$. This implies the following approximation:

%\begin{prop}\cite{besson1995entropies}
%Let $p\geq 3$.  For all $H\in \Sym_p^{++}, \, \Tr(H)=1$,  
%$$\text{if }\,  \frac{(\det H)^{1/2}}{\det(\Id-H)}=\frac{p^{p/2}}{(p-1)^p}, \, \text{then }\,  H=\frac{1}{p}\Id.$$
%\end{prop}
%I
%\begin{prop}\label{lem - approx BCG 2}
%Let $p\geq 3$. $\forall \eta >0, \, \exists \epsilon_0>0, \, \forall 0<\epsilon\leq \epsilon_0,\,  \forall H\in \Sym_p^{++},  \, \Tr(H)=1\, $
%$$\text{ if } \, 
%\left|\frac{(\det H)^{1/2}}{\det(\Id-H)}-\frac{p^{p/2}}{(p-1)^p}\right|\leq \epsilon, \,\text{ then }\,
%\|H-\frac{1}{p} \Id\|\leq \eta.$$
%\end{prop}
\begin{lemme}\label{lem - approx BCG 2}
$\forall \eta>0, \, \exists \epsilon_0>0, \, \forall 0<\epsilon\leq \epsilon_0,\,  \text{if } \, \left|\Jac F_j -1 \right|\leq \epsilon \text{ then } $
$$\left\|\tilde{H}_{y,j}-\frac{1}{m} \Id\right\|\leq \eta.$$
\end{lemme}

\section{Proof of Theorem \ref{th principal}}
%We are going to follow the work of Besson, Courtois and Gallot, in order to construct a sequence of $(\G,\rho_j(\G)-$equivariant maps, $F_j$, which will converge uniformly to an isometry $F$. 
We now show that up to subsequence and conjugacy $F_j$ converges. We follow the main steps presented in the paragraph 4 of \cite{besson2007inegalites}. 

\paragraph{Step 1 : Almost everywhere convergences  of $\tilde{H}_{y,j}$.}
\begin{lemme}\label{lem - Jac F k converges almost everywhere}
Up to subsequence, $\Jac(F_j(y))$  converges almost everywhere to $1$. 
\end{lemme}
\begin{proof}
%It is proven in \cite[p154]{besson1999lemme} that 
%$$\Jac_m(F_j(y)) \leq \left(\frac{\delta(\G)}{m-1}\right)^m=1,$$
%the last equality follows from the fact that $\G$ is a cocompact lattice of $Y$, and therefore $\delta(\G) = h_{vol}(Y) = m-1$.
%
%We can also consider the  map $G_j \, :\, X\tv Y$ which is constructed in the same manner as $F_j$, and gives a $(\rho_j(\G), \G)-$equivariant map. This map will satisfies for all $x\in X$, and all $3\leq p\leq n-1$
%$$\Jac_p(G_j(x)) \leq \left(\frac{\delta(\rho_j(\G))}{p-1}\right)^p.$$

Applying the same construction and inequalities in the direction $X\tv Y$ we get a sequence of $(\rho_j(\G),\G)$-invariant maps, $G_j$ satisfying, for all $ p\in [3,n]$
\begin{eqnarray}
\Jac_p G_j(x) \leq \frac{\delta(\rho_j(\G))^p}{(p-1)^p}
\end{eqnarray}

The maps $H_j=G_j\circ F_j \, :\, Y\tv Y$ are $\G$-invariant, of degree $1$ and satisfies $\Jac(H_j(y))=\Jac_m(H_j(y)) \leq \left(\frac{\delta(\rho_j(\G))}{m-1}\right)^m $. Therefore: 
$$\vol(Y/\G) =\int_{Y/\G} \Jac_m(H_j(y))dy \leq \left(\frac{\delta(\rho_j(\G))}{m-1}\right)^m \vol(Y/\G).$$

Since $\delta(\rho_j(\G) )\overset{j\tv \infty}{\tv }m-1$, this implies that up to subsequence, $\Jac_m(H_j(y))$  converges almost everywhere to $1$. Let $x=F_j(y)$, since $\Jac_m(H_j(y))\leq\Jac_m(G_j(x))\Jac(F_j(y))$ it implies that $\Jac(F_j)\tv 1$ almost everywhere. 
\end{proof}

We will still denote this converging subsequence by the index $j$.

Using Lemmas \ref{lem - approx BCG 1} and \ref{lem - approx BCG 2}, the previous result shows: 
\begin{lemme}
For almost every $y\in Y$, $\lim_{j\tv \infty} \tilde{H}_{y,j} =\frac{1}{m}\Id$. 
\end{lemme}

As in \cite{besson2007inegalites}, we will denote by $\tilde{H}_0$ the quadratic form on $dF_j(T^1_y(Y))\subset T^1_{F_j(y)} X$ equal to $\frac{1}{m}\Id$. 

\paragraph{Step 2 : Uniform convergence of $\tilde{H}_{y,j}$ to $\tilde{H}_0$.}
We denote by $\mu_j$ the largest eigenvalue of $\tilde{H}_{y,j}$. 
\begin{lemme}\cite[Lemma 4.7]{besson2007inegalites}\label{lem conve unif lem 1}
 Let $y,y'$ be two points in $Y$ such that $\mu_j \leq 1-\frac{1}{m}$ on every points of the geodesic from $y$ to $y'$, then there exists a constant $C$ such that 
$$d_X(F_j(y),F_j(y') ) \leq C d_Y(y,y').$$
\end{lemme}
Since our setting is a bit simpler that the one in \cite{besson1995entropies,besson2007inegalites}, we make a proof without the technical problems that appears therein.
\begin{proof}
%Using the definition of the barycenter by critical point of the functional $\mathcal{B}$ we have 
%\begin{eqnarray*}
%g_0(K_{y,j} dF_j(y) u,v)& =&   \int_{\partial Y} Dd\beta_X\left|_{(F_j(y),f_j(\xi))} (DF_j(y)u ,v) d\nu_y(\xi)\right.\\
%&=& E_Y\int_{\partial Y} d\beta_X \left|_{(F_j(y),f_j(\xi))}(v) d\beta_Y \left|_{(y,\xi)} (u)d\nu_y(\xi) \right.\right.\\
%&\leq & E_Y h_{y,j} (v,v)^{1/2} h'_y(u,u)^{1/2}\\
%&\leq & CE_Y \sqrt{\mu_j(y)}
%\end{eqnarray*}
Recall Inequality (\ref{eq - Cauchy Schwarz})
$$g_0(K_{y,j} dF_j(y) u,v) \leq (m-1) h_{y,j} (v,v)^{1/2} h'_y(u,u)^{1/2}$$
We already remarked that $\|H'_{y,j}\|$ is bounded independently of $j\in \N$ and $y\in Y$. Therefore there exists $C_1>0$ such that 
$$g_0(K_{y,j} dF_j(y) u,v) \leq C_1 h_{y,j} (v,v)^{1/2}\leq C_1  \sqrt{\mu_j(y)}.$$

Moreover, by hypothesis we have 
$\tilde{K}_{y,j}=\Id - \tilde{H}_{y,j} \geq (1-\mu_j(y)) \Id$. Therefore, by taking $v=\frac{dF_j(y)}{\| dF_j(y)\|},$ we have 
$$(1-\mu_j(y)) \|dF_j(y)\|_{g_0} \leq  C_1\sqrt{\mu_j(y)}$$
Let $\alpha(t)$ be the geodesic joining $y$ and $y'$. The last inequality, implies that there exists $C$ independent of $j\in \N$ and $u\in T^1_{\alpha(t)}Y$ such that, 
$\|dF_j(u)\|_{g_0} \leq C$. 
The lemma follows thanks to the mean value inequality. 

\end{proof}

The following lemma can also be found in \cite{besson1995entropies,besson2007inegalites} with a function $F_j$ slightly more complicated.  
\begin{lemme}\cite{besson2007inegalites}\label{lem conv unif 2}
With the same notations as in the previous lemma, let $P$ denotes the parallel transport from $F_j(y) $ to $F_j(y')$ along the geodesic in $X$ joining these two points. We have 
$$\| h_{y,j} - h_{y',k}\circ P\| \leq 4 d_X(F_j(y),F_j(y')).$$ 
\end{lemme}
\begin{proof}
Let $\beta(t)$ be the geodesic from $F_j(y) $ to $F_j(y')$. Let $Z$ be a unit parallel vector field along $\beta$, and called $Z_1 = Z(F_j(y))$ and $Z_2=Z(F_j(y'))$. We then have: 
\begin{eqnarray*}
h_{y',k}(Z_2,Z_2)-h_{y,j}(Z_1,Z_1) &= &\int_{\partial X} \left(d\beta_X \left|_{(F_j(y'),f_j(\xi))}(Z_2)\right)^2 d\xi \right.-\int_{\partial X} \left(d\beta_X \left|_{(F_j(y),f_j(\xi))}(Z_1)\right)^2 d\xi\right.\\
\end{eqnarray*}
On one hand, we have: 
$$\Big|d\beta_X \left|_{(F_j(y'),f_j(\xi))}(Z_2)\right. - d\beta_X \left|_{(F_j(y),f_j(\xi))}(Z_1)\right. \Big|\leq \left( \sup_{t}\left| Dd\beta_X\left|_{\beta(t)} \right. (\dot{\beta   } ,Z)\right| \right) d_X (F_j(y),F_j(y')).$$
Since $Z$ is unitary,  $Dd\beta_x= g_X - d\beta_X\otimes d\beta_X$, and 
$\| d\beta_X \left|_{(F_j(y'),f_j(\xi))} (\cdot) \|  \leq 1\right.$ we have
$$\left( \sup_{t}\left| Dd\beta_X\left|_{\beta(t)} \right. (\dot{\beta   } ,Z)\right| \right) \leq 2.$$
On the other hand, for any unitary vector $u,v$ we have: 
$$\Big| d\beta_X \left|_{x,\xi} (u)\right. +d\beta_X \left|_{x,\xi} (v) \right.\Big| \leq 2.$$
Therefore,
$$\left|\left(d\beta_X \left|_{(F_j(y'),f_j(\xi))}(Z_2)\right.\right)^2 - \left(d\beta_X \left|_{(F_j(y),f_j(\xi))}(Z_1)\right.\right)^2 \right|\leq 4 d_X(F_j(y),F_j(y')). $$

\end{proof}

\begin{lemme}\cite[Lemma 4.9]{besson2007inegalites}
$\tilde{H}_{y,j}$ converges uniformly to $\tilde{H}_0$ as $j\tv \infty$. 
\end{lemme}
The proof is exactly the same as in \cite[Lemme 4.9]{besson2007inegalites}. It uses  Egoroff's theorem, Lemmas \ref{lem conve unif lem 1} and  \ref{lem conv unif 2}, but does not use the particular form of $F_j$. 
\bigskip \\
\textbf{Step 3 : Uniform convergence of $F_j$.}\\

\noindent The following lemma, corresponds to Lemma 4.10 in \cite{besson2007inegalites}. 
\begin{lemme}
Up to subsequence and composition by an element of $\Isom (X)$, $F_j$ converges uniformly to a  continuous map $F \, :\, Y\tv X$.
\end{lemme}
\begin{proof}
For all $\epsilon >0$, there exists $j\gg 0$ such that for all $y\in Y$ we have 
\begin{equation*}
\tilde{H}_{y,j} \leq  \tilde{H}_0 +\epsilon \Id
\quad \text{and} \quad \tilde{K}_{y,j} \geq \tilde{K}_0 -\epsilon \Id.
\end{equation*}
Therefore, using Inequality (\ref{eq - Cauchy Schwarz}) there exists $C>0$ independent of $j\in \N$ such that, for all $u\in T^1 Y$: 
$$\| dF_j (u) \|_{X} \leq C.$$
This means that the sequence $F_j$ is $C-$Lipschitz. %In order to apply Ascoli's theorem we need to show that for all $y\in Y$ the set $\{F_j(y)\}_{j\in\N} $ is relatively compact. 
We fix $y_0$ in $Y$ and $x_0\in X$. Let  $\g_j\in \Isom(\Hyp^m)$ be an element such that $\g_j F_j(y_0)=x_0$, and call $F'_j =\g_j \circ F_j$.
Since $\g_j$ is an isometry of $X$ we have  $\| dF'_j (u) \|_{X} \leq C.$ Now for any point $y\in Y$ we have 
$$d_X(F'_j(y), F'_j(y_0) ) \leq C d_Y(y,y_0).$$
Since $F'_j(y_0)$ is chosen to be equal to $x_0$, $\{F'_j(y)\}_{j\in\N} $ is bounded, and we conclude by Ascoli's theorem. 
\end{proof}

Let $F$ be a uniform limit of $F'_j$.  \\

\begin{lemme}
The sequence $\rho_j \, :\, \G \tv \Isom(\Hyp^n)$ admits a subsequence which converge algebraically to  a discrete and faithfull representation $\rho_\infty$. 
\end{lemme}

\begin{proof}
For every $\g\in\G$, the sequence $\rho_j(\g)$ is equicontinuous since they are $1-$Lipschitz maps.
Let $j\gg 0$ such that for all $y\in Y$, $d(F_j(y),F(y)) \leq \epsilon$. Now, take any point $x\in X$
\begin{eqnarray*}
d(\rho_j(\g)x,x) &\leq & d(\rho_j(\g) x, \rho_j(\g) F (y)) + d(\rho_j(\g) F (y) , F (\g y ) ) +d(F (\g y) ,x)\\
&\leq & d(x, F (y)) +d(F(\g y ) , x)+d(\rho_j(\g) F(y ) , \rho_j(\g) F_j(y)) +d(\rho_j(\g) F_j(y), F(\g y)\\
&\leq & d(x, F (y)) +d(F(\g y ) , x)+d(F(y),F_j(y)) +d(F_j(\g y ) ,F(\g y )) \\
&\leq &  d(x, F (y)) +d(F(\g y ) , x)+2\epsilon.
\end{eqnarray*}

Therefore $\{\rho_j(\g) x \, | \, k\in \N\}$ is relatively compact for all $x\in X$ and all $\g \in \G$. Ascoli's Theorem asserts that $\rho_j(\g)$  admits a converging subsequence, call it $\rho_\infty (\g)$.  We can make a diagonal argument on a finite set of generators to find a subsequence still denoted $\rho_j$ for which all $\rho_{j}(\g_i) $ converges to some $\rho_\infty(\g_i)$, where  $\G =<\g_i, \, i\in [1,r]>.$ %Writing any element $\g=\g_{i_1}\dots \g_{i_s} $ as a product of $\g_i$ we see that for any element  $\g\in \G$,  
%$\rho_{\phi(k)}(\g)$ converges. Since relations in the group are continuous, the limit $\rho_\infty(\g_i)$ preserves the relations, and therefore $\rho_\infty(\g)$ is independent of the way the product $\g=\g_{i_1}\dots \g_{i_s} $ is chosen. 
By definition this means that $\rho_{j}$ converges to a representation $\rho_\infty$. 

Now we use the result cited in the introduction due to Kapovich \cite[Theorem 1.1]{kapovich2008sequences} which implies that $\rho_\infty$ is discrete and faithful. 

\end{proof}

\bibliographystyle{alpha}

\end{document}